\documentclass[12pt]{amsart}
\usepackage{amsmath,amsfonts,amssymb} \usepackage{color}
\usepackage[all,cmtip]{xy}
\usepackage{graphicx}
 \usepackage{youngtab}

\newcommand{\corr}[1]{\langle {#1} \rangle}

 \newcommand{\bR}{\mathbb{R}}  
 
\newcommand{\bt}{{\bf t}}

 \newcommand{\pd}{\partial}
\newcommand{\Mbar}{\overline{\mathcal M}}

\newcommand{\be}{\begin{equation}}
\newcommand{\ee}{\end{equation}}
\newcommand{\bea}{\begin{eqnarray}}
\newcommand{\eea}{\end{eqnarray}}
\newcommand{\ben}{\begin{eqnarray*}}
\newcommand{\een}{\end{eqnarray*}}

\newcommand{\half}{\frac{1}{2}}

\newtheorem{cor}{Corollary}[section]

\newtheorem{lem}[cor]{Lemma}
 \newtheorem{prop}[cor]{Proposition}
 \newtheorem{thm}[cor]{Theorem}
\theoremstyle{remark}

\definecolor{A}{rgb}{.75,1,.75}

\definecolor{green}{rgb}{0,1,0}
\definecolor{yellow}{rgb}{1,1,0}
\definecolor{orange}{rgb}{1,.7,0}
\definecolor{red}{rgb}{1,0,0}
\definecolor{white}{rgb}{1,1,1}

 \makeindex
\begin{document}
\title
{On a Mean Field Theory of Topological 2D Gravity}

\author{Jian Zhou}
\address{Department of Mathematical Sciences\\Tsinghua University\\Beijng, 100084, China}
\email{jzhou@math.tsinghua.edu.cn}

\begin{abstract}
We present a one-dimensional mean field theory for topological 2D gravity.
We discuss possible generalizations to other topological field theories,
in particular those related to 
semisimple Frobenius manifolds. 
\end{abstract}

\maketitle

\section{Introduction}

The mathematical theory of the topological 2D gravity studies the following
intersection numbers on the Deligne-Mumford moduli spaces:
\be
\corr{\tau_{a_1} \cdots \tau_{a_n}}_g := \int_{\Mbar_{g,n}} \psi_1^{a_1} \cdots \psi_n^{a_n}.
\ee
As is well-known,
instead of considering these numbers individually,
it is more effective to study them collectively by  considering their formal generating series:
\be  \label{def:F}
F(\bt; \lambda) = \sum_{g \geq 0} \lambda^{2g-2} F_g(\bt),
\ee
where
\be
F_g(\bt) : = \sum  \corr{\tau_{a_1} \cdots \tau_{a_n}}_g \frac{t_{a_1} \cdots t_{a_n}}{n!}.
\ee

Apparently $F$ involves  integrations on infinitely many spaces and infinitely many parameters.
In this paper,
we will present a mean field theory for this theory which is one-dimensional,
more precisely,
a model that depends on formal integration of a formal field on a one-dimensional
space, depending on infinitely parameters.  
This is very striking because the original theory involves 
all sorts of topology of Riemann surfaces and their suitable compactifications, 
it is very hard to expect that all the information can be essentially 
encoded in a one-dimensional theory.
The secret is that we hide all the complexities in the interactions of the mean
field theory. 

The idea of a mean field theory is a very well-known old idea in statistical physics where one considers 
a system with a very large degree of freedom. 
One approximates the system by finding a field theory for suitably chosen 
finitely many order parameters 
so that the Euler-Lagrange equations approximate the equations of states of the systems.
Usually quantum corrections are needed to improve the approximation. 
Renormalization flow was developed to provide a scheme to canonically improve the model. 

The application of the approach of mean field theory to topological string theory
was initiated in \cite{Dijkgraaf-Witten}.
The authors focus on  the genus zero case based on 
topological recursion relations in genus zero (see also \cite{Itzykson-Zuber} for mean field theory in genus zero
for topological 2D gravity).
They also discussed the cases of genus one and higher genera based on integrable hierarchies.
Their discussions are mostly concerned with the derivation of the equations of states.
In this paper we will supplement their work by providing the suitable action
functional that leads to these equations.
We will focus on the case of pure topological gravity. 
We will speculate on the generalizations to the general case in the final Section 5
and leave the details to be worked out in subsequent work.

For the order parameter in the case of topological 2D gravity,
we take the specific heat defined as the second derivative of the free energy: 
\be
u(\bt; \lambda) = \sum_{g \geq 0} \lambda^{2g} u_g(\bt) 
= \sum_{g \geq 0} \lambda^{2g} \frac{\pd^2F_g(\bt)}{\pd t_0^2}. 
\ee 
It is customary to take $t_0$ as the space variable and denote it by $x$,
$u$ then can be thought of as field on the space $\bR^1$ with coordinate $x$,
parameterized by time variables $t_1, t_2, \dots$.
By Witten Conjecture/Kontsevich Theorem \cite{Witten, Kontsevich},
$u$ satisfies the KdV hierarchy
\be
\frac{\pd u}{\pd t_n} = \pd_{x} R_n[u],
\ee
for the sequence of Gelfand-Dickey differential polynomials $R_n[u]$. 
It is known that there is a differential polynomial $T_n[u]$ such that
\be
\pd_x T_n[u] = \pd_x u \cdot R_n[u].
\ee
The action we find for our mean field theory is
\be
S[u] =  \frac{1}{\lambda^2}
\sum_{n \geq 0} (t_n - \delta_{n1}) \cdot \int  T_{n}[u(x)] dx +  F(\bt)|_{t_0=0}.
\ee
We will show that its Euler-Lagrange equation is equivalent to the string equation:
\be  
u = \sum_{k=1}^\infty t_n  R_n[u] + x.
\ee
Furthermore, if $u(\bt)$ is determined by this equation,
then we have
\be
F(\bt) =  \frac{1}{\lambda^2}
\sum_{n \geq 0} (t_n - \delta_{n0}) \cdot \int  T_{n}[u(\bt)] dx +  F(\bt)|_{t_0=0}.
\ee
This means our mean field theory is exact semiclassically without any quantum corrections.

Note in the above we have not specified the bounds for integration.
This is because we are working in a formal setting without considering the issue of convergence. 
In fact, we will show that we cannot expect the convergence in this case.
Since $F_g(\bt)$ involves infinitely many variables,
to make sense of the convergence problem
one can restrict to the subspaces where
only fixed  numbers of finitely many variables are possibly nonvanishing.
In fact we will consider the restriction to the space where all variables expect for $t_0$
and $t_2$ vanish.
Denote by $F_g(t_0, t_2)$ the restriction of $F_g(\bt)$ on this space.
We will establish the following result:
\be
\sum_{g \geq 0} u_g(t_0, t_2) \lambda^{2g}
= \frac{1}{t_2} + \sum_{g \geq 0} \frac{2 a_g}{24^g} t_2^{3g-1} (1-2t_0t_2)^{-(5g-1)/2} \lambda^{2g},
\ee
where $\{a_g\}_{g\geq 0}$ is a sequence of integers studied in probability theory of graphs
 \cite{Janson, Janson-Chassaing},
see also \cite{Finch}.
By the asymptotic formula for $a_n$,
we have
\be
\lim_{n \to \infty} \frac{1}{a_n^{1/n}} = 0.
\ee
It follows that when $u_2 \neq 0$,
the radius of convergence of the series
\be
\sum_{g \geq 0} F_g(t_0, t_2) \lambda^{2g}
\ee
as a power series in $\lambda$ is zero.

The rest of the paper is arranged as follows.
In Section 2 we recall the Gelfand-Dickey polynomials $R_n[u]$ 
and some related differential polynomials $T_n[u]$ 
and establish some variational properties of $T_n$.
The mean field theory of the topological 2D gravity is derived 
in Section 3 based on the string equation.
The convergence problem in Section 4 will be addressed in Section 3. 
In the final Section 5 we present some speculations about possible generalizations.

\section{Gelfand-Dickey Polynomials and Their Properties}

\subsection{Gelfand-Dickey polynomials}

Following \cite{Gelfand-Dikii},
define a sequence $\{ R_n\}$ of differential polynomials in $u$
by the Lenard recursion relations:
\be \label{eqn:Lenard}
\begin{split}
& R_0 = 1, \\
& \pd_x R_{n+1} = \frac{1}{2n+1} \biggl(\pd_xu \cdot R_n
+ 2u \cdot \pd_x R_n + \frac{\lambda^2}{ 4} \pd_x^3R_n \biggr).
\end{split}
\ee
For example,
\ben
\pd_x R_1 & = & \pd_x u, \\
R_1 & = & u, \\
\pd_x R_2 & = & u \cdot \pd_x u + \frac{\lambda^2}{12} \pd_x^3 u, \\
R_2 & = & \frac{1}{2} u^2 + \frac{\lambda^2}{12} \pd_x^2u,  \\
\pd_x R_3 & = & \frac{1}{2} u^2\cdot \pd_x u + \frac{\lambda^2}{12} u \cdot \pd_x^3u
+ \frac{\lambda^2}{6}\pd_x u \cdot \pd_x^2u + \frac{\lambda^4}{240}\pd_x^5u, \\
R_3 & = & \frac{1}{6} u^3 + \frac{\lambda^2}{12} u \cdot \pd_x^2 u
+ \frac{\lambda^2}{24} (\pd_x u)^2 + \frac{\lambda^4}{240} \pd_x^4 u.
\een

Rewrite \eqref{eqn:Lenard} as follows:
\be \label{eqn:Lenard-2}
 R_{n+1} = \frac{1}{2n+1} \biggl(u \cdot  + \frac{\lambda^2}{4} \pd_x^2 \biggr) R_n
+ \frac{1}{2n+1} \pd_x^{-1}(u \cdot \pd_x R_n),
\ee
or alternative as
\be \label{eqn:Lenard-3}
 R_{n+1} = \frac{1}{2n+1} \biggl(2u \cdot  + \frac{\lambda^2}{4} \pd_x^2 \biggr) R_n
- \frac{1}{2n+1} \pd_x^{-1}(\pd_x u \cdot  R_n),
\ee
To find $R_{n+1}$, one needs to show that $u \cdot \pd_x R_n$ or $\pd_x u \cdot R_n$ is a total derivative
and finds the corresponding antideritive.
More generally, in \cite{Gelfand-Dikii} it was proved that for $k, l \geq 0$,
there exists a differential polynomial $P_{k,l}$ such that
\be
R_k \cdot \pd_x R_l = \pd_x P_{k,l}.
\ee
In particular,
there are differential polynomials $T_n$ such that
\be \label{Def:T-n}
\pd_x u \cdot R_n = \pd_x T_n.
\ee
The following are the first few terms:
\ben
&& T_0 = u, \\
&& T_1 = \frac{1}{2} u^2, \\
&& T_2 = \frac{1}{6} u^3+ \frac{1}{24} u_x^2\lambda^2, \\
&& T_3 = \frac{1}{24} u^4 + \frac{1}{24} u u_x^2 \lambda^2
- \frac{1}{480} u_{2x}^2\lambda^4+ \frac{1}{240} u_x u_{3x}\lambda^4, \\
&& T_4 = \frac{1}{120} u^5+ \frac{1}{48} u^2 u_x^2\lambda^2
+ \frac{1}{240} u_x^2u_{2x} \lambda^4  - \frac{1}{480} u u_{2x}^2 \lambda^4 \\
&& + \frac{1}{240} u u_x u_{3x} \lambda^4
+ \frac{1}{13440} u_{3x}^2\lambda^6
- \frac{1}{6720} u_{2x} u_{4x} \lambda^6
+ \frac{1}{6720} u_x u_{5x} \lambda^6.
\een

\subsection{Variational derivative}

Denote by $A = \oplus_{n \geq 0} A_n$  the space of differential polynomials,
i.e.,
polynomials in $u_0=u(x), u_1 = \pd_x u(x)$, $\dots$, $u_n = \pd_x^n u(x)$, $\dots$.
By $A_n$ we denote the space of homogeneous differential polynomials of degree $n$.
On the space $A$ the operators $\frac{\pd}{\pd u_k}$ naturally act,
so do the operator
\be
\pd_x = \sum_{k \geq 0} u_{k+1} \frac{\pd}{\pd u_k}
\ee
and the operator $\delta$ defined as follows:
\be
\delta = \sum_{k \geq 0} (-1)^k \pd_x^k \frac{\pd}{\pd u_k}.
\ee
For $f \in A$, $\delta f$ will be called the variational derivative of $f$.
Since
\be
[\frac{\pd}{\pd u_k}, \pd_x] = \begin{cases}
0, & \text{if $k=0$}, \\
\frac{\pd}{\pd u_{k-1}}, & \text{if $k \geq 1$},
\end{cases}
\ee
one can easily see that
\be
\delta \pd_x = 0.
\ee
In \cite{Gelfand-Dikii} the following sequence
\be
0 \to \bR \to A \stackrel{\pd_x}{\to} A \stackrel{\delta}{\to} A
\ee
is proved to be exact.
Furthermore,
the following identity is established by symbolic computations:
\be \label{eqn:Delta-U1}
\delta (u_1 \cdot  \delta f )= 0,
\ee
for $f \in A$.
By slightly modifying the proof,
one can also prove the following identity:
\be
\delta (u \cdot  \delta f )=  n\cdot  \delta f,
\ee
for $f \in A_n$.

\subsection{Variational property of $T_n$}

In \cite{Gelfand-Dikii},
the following relation is proved:
\be
\delta R_{n+1} = R_n.
\ee
By \eqref{eqn:Delta-U1} one then has
\be
\delta (\pd_x u \cdot R_{n}) = 0.
\ee

\begin{lem}
For $n \geq 0$,
\be
\delta(u R_n) = (n+1) R_n - \frac{1}{2} \lambda \frac{\pd}{\pd \lambda} R_n.
\ee
\end{lem}

\begin{proof}
Write  $R_n = \sum_{g\geq 0} \lambda^{2g} R_n^{(g)}$, where $R_n^{(g)} \in A$.
By \eqref{eqn:Lenard} it is not hard to see that
\be
R_n^{(g)} \in A_{n - g}.
\ee
So we have
\ben
\delta (uR_n) 
& = & \delta (u \cdot R_n)    = \delta (u \cdot \delta R_{n+1})   \\
& = & \sum_{g\geq 0} \lambda^{2g} \delta (u \cdot \delta R^{(g)}_{n+1})  \\
& = & \sum_{g\geq 0} \lambda^{2g} (n+1-g) \delta R^{(g)}_{n+1}  \\
& = & (n+1 - \half \lambda \frac{\pd}{\pd \lambda}) R_n.
\een
\end{proof}

The following result plays a key role in the next Section:

\begin{prop}
For $n \geq 0$,
the following holds:
\be
 \delta  T_n = (1 -\lambda \frac{\pd}{\pd \lambda}) R_n.
\ee
\end{prop}

\begin{proof}
Rewrite \eqref{eqn:Lenard-3} as follows:
\be 
T_{n} =  2u \cdot  R_n + \frac{\lambda^2}{4} \pd_x^2  R_n
- (2n+1)  R_{n+1} ,
\ee
Apply $\delta$ on both sides:
\ben
\delta T_n 
& = & 2 \delta (u \cdot R_n) + \frac{\lambda^2}{4} \delta \pd_x^2R_n - (2n+1) \cdot \delta R_{n+1}  \\
& = & 2(n+1 - \half \lambda \frac{\pd}{\pd \lambda}) R_n - (2n+1) R_n \\ 
& = & (1 -\lambda \frac{\pd}{\pd \lambda}) R_n.
\een

\end{proof}

\section{Mean Field Theory of Topological 2D Gravity}

\subsection{The KdV equations}

The KdV hierarchy is the following sequence of partial differential equations:
\be \label{eqn:KdV}
\pd_{t_n} u = \pd_{t_0} R_{n+1},
\ee
where $t_0 = x$.

\subsection{The string equation}

By the {\em  puncture equation} we mean the following equation:
\be
\frac{\pd F}{\pd  t_{0}}
= \sum_{k=1}^\infty t_k\frac{\pd F}{\pd t_{k-1}} + \frac{t_0^2}{2\lambda^2},  \label{eqn:Puncture} \\
\ee
In the mathematical literature this is referred to as the {\em string equation}.
Following physicists,
we will reserve this name for the equation \eqref{eqn:String} below.
Take $\lambda^2\pd_{u_0}^2$ on both sides of \eqref{eqn:Puncture}:
\be
 \frac{\pd u}{\pd  t_0}
= \sum_{k=1}^\infty t_k\frac{\pd u}{\pd t_{k-1}} + 1.
\ee
And so by \eqref{eqn:KdV},
\be
\pd_x u= \sum_{k=1}^\infty t_k \pd_x R_k + 1.
\ee
It can be shown that when $u= \pd_x^2F$ satisfies the KdV hierarchy
and $F$ satisfies the  puncture equation,
the following equation holds \cite{DVV, Getzler, Liu}:
\be \label{eqn:String}
u = \sum_{k=1}^\infty t_k  R_k + x.
\ee
This is called the {\em string equation} in the physics literature.

\subsection{Landau-Ginzburg equation}

Expanding both sides of the string equation as series in $\lambda$,
one gets by comparing the leading terms the following equation:
\be \label{eqn:Landau-Ginzburg}
u_{0} = \sum_{k=1}^\infty \frac{1}{k!} t_k \cdot u_{0}^k + x.
\ee
We will call this equation  the {\em Landau-Ginzburg equation}.

\subsection{Derivation of the mean field theory}

Multiply both sides of \eqref{eqn:String} by $u_x$ and integrate with respect to $x$:
\be
\int u u_x dx = \sum_{k=1}^\infty t_k  \int u_x R_k dx + \int x u_x dx.
\ee
The last term on the right-hand side can be found by integration by parts:
\ben
\int x u_x dx  = \int x du = x u - \int u dx = xu - \lambda^2 \frac{\pd F}{\pd x}
\een
So we get
\be
\half u^2  = \sum_{k=1}^\infty t_k T_k + x u - \lambda^2 \frac{\pd F}{\pd x}.
\ee
Rewrite it as follows:
\ben
\frac{\pd}{\pd t_0} F & =  &  \frac{1}{\lambda^2} \sum_{n \geq 0} (t_n - \delta_{n1}) \cdot T_n.
\een
Integrate once more,
we obtain the following

\begin{prop}
Suppose that $u=u(\bt)$ is determined by \eqref{eqn:String},
then one has:
\be
F(\bt) =  F(\bt)|_{t_0=0} + \frac{1}{\lambda^2}
\sum_{n \geq 0} (t_n - \delta_{n1}) \cdot \int  T_{n}[u(\bt)] dx,
\ee
\end{prop}

\begin{thm}
Define the following action for a formal field $u=u(x)$ on $\bR^1$:
\be
S[u] =  F(\bt)|_{t_0=0} + \frac{1}{\lambda^2}
\sum_{n \geq 0} (t_n - \delta_{n1}) \cdot \int  T_{n}[u(x)] dx.
\ee
Then the Euler-Lagrange equation for this action is equivalent to the string equation 
\eqref{eqn:String}.
\end{thm}

\begin{proof}
This is because
\ben
\frac{\delta}{\delta u} S[u] 
& = & \frac{1}{\lambda^2}
\sum_{n \geq 0} (t_n - \delta_{n1}) \cdot \delta T_{n}[u(x)] \\
& = & \frac{1}{\lambda^2}
\sum_{n \geq 0} (t_n - \delta_{n1}) \cdot (1-\lambda \frac{\pd}{\pd \lambda}) R_{n}[u(x)] \\
& = &  
\sum_{n \geq 0} (t_n - \delta_{n1}) \cdot 
\sum_{g \geq 0} \lambda^{2g-2} (1-2g) R_{n}^{(g)}[u(x)].
\een
Therefore, 
if $\frac{\delta}{\delta u} S[u]  = 0$,
then one has for all $g\geq 0$,
\ben
(1-2g) 
\cdot \sum_{g \geq 0} \sum_{n \geq 0} (t_n - \delta_{n1}) \cdot  R_{n}^{(g)}[u(x)] = 0,
\een
and so 
\ben
\sum_{g \geq 0} \sum_{n \geq 0} (t_n - \delta_{n1}) \cdot  R_{n}^{(g)}[u(x)] = 0.
\een
It follows that \eqref{eqn:String} holds.
Conversely,
if $u$ satisfies \eqref{eqn:String},
then it also satisfies the Euler-Langrage equation.
\end{proof}

By the explicit expression for the first few $T_n$'s given earlier,
we see that the first few terms of the Lagrangian for our action functional is:
\ben
L &= & t_0 u + \frac{t_1-1}{2} u^2 + t_2 (\frac{1}{6} u^3+ \frac{1}{24} u_x^2\lambda^2) \\
& + & t_3 (\frac{1}{24} u^4 + \frac{1}{24} u u_x^2 \lambda^2
- \frac{1}{480} u_{2x}^2\lambda^4+ \frac{1}{240} u_x u_{3x}\lambda^4) \\
& + & t_4 ( \frac{1}{120} u^5+ \frac{1}{48} u^2 u_x^2\lambda^2
+ \frac{1}{240} u_x^2u_{2x} \lambda^4  - \frac{1}{480} u u_{2x}^2 \lambda^4 \\
&+ & \frac{1}{240} u u_x u_{3x} \lambda^4
+ \frac{1}{13440} u_{3x}^2\lambda^6
- \frac{1}{6720} u_{2x} u_{4x} \lambda^6
+ \frac{1}{6720} u_x u_{5x} \lambda^6) + \cdots.
\een
The genus zero part of the Lagrangian is:
\be
L_0 = \sum_{n \geq 0} (t_n -\delta_{n1}) \frac{u^{n+1}}{(n+1)!}.
\ee
This is what we used in \cite{Zhou2} for the mean field theory of the topological 1D gravity.
Taking all $t_n = 0$,
one gets a plane algebraic curve:
\be
L_0 = - \half u^2.
\ee
This is the Airy curve \cite{Bennett-Cochran-Safnuk-Woskoff, Eynard, Zhou1}
that determines the topological 2D gravity by Eynard-Orantin 
topological recursion \cite{EO}.

\section{Convergence Problem}

Now we restrict to the $(t_0, t_2)$-plane.
The string equation \eqref{eqn:String} becomes:
\be
u(t_0, t_1) = t_0 + t_2 \biggl(\frac{1}{2} u^2(t_0, t_1) + \frac{\lambda^2}{12} \pd_{t_0}^2u(t_0, t_1) \biggr).
\ee
We solve the above equation recursively by rewriting it as follows:
\be \label{eqn:u0}
u_0(t_0, t_1) = t_0 +  \frac{t_2}{2} u_0^2(t_0, t_1),
\ee
and for $g \geq 1$,
\be \label{eqn:ug}
u_g(t_0, t_2)
= \frac{t_2}{2}  \sum_{g_1+g_2=g} u_{g_1}(t_0, t_2) \cdot u_{g_2}(t_0, t_2)
+ \frac{t_2}{12} \frac{\pd^2}{\pd t_0^2} u_{g-1}(t_0, t_2).
\ee
One can rewrite \eqref{eqn:ug} in the following form:
\ben
u_g(t_0, t_2) = \frac{t_2}{1- t_2u_0(t_0, t_2)} \biggl(\half \sum_{g_1=1}^{g-1} u_{g_1}(t_0, t_2) \cdot u_{g-g_1}(t_0, t_2)
+ \frac{1}{12} \frac{\pd^2}{\pd t_0^2} u_{g-1}(t_0, t_2) \biggr).
\een
From \eqref{eqn:u0}
one can get \cite{Zhou}:
\be
u_0(t_0, t_2)  =  \frac{1- ( 1 - 2t_0t_2 )^{1/2} }{t_2}.
\ee
After using \eqref{eqn:ug} to get
\bea
u_1(t_0, t_2)  =  \frac{1}{12} t_2^2 ( 1 - 2t_0 t_2 )^{-2}, \\
u_2(t_0, t_2)  = \frac{49}{288} t_2^5 ( 1 - 2t_0t_2 )^{-9/2},
\eea
we make the following ansatz:
\be
u_g = c_g t_2^{3g-1} (1-2t_0t_2)^{-(5g-1)/2} + \delta_{g,0} t_2^{-1}
\ee
and get the following recursion relations for the coefficients $c_g$:
\be
c_g = \frac{1}{2} \sum_{g_1=1}^{g-1} c_{g_1} c_{g-g_1}
+ \frac{1}{12} (5g-4)(5g-6) c_{g-1}.
\ee
Set $c_g = \frac{2}{24^g} a_g$,
then one has
\be
a_0 = - \frac{1}{2},
\ee
and for $n > 0$ the recursion relation:
\be
a_n =  \sum_{k=1}^{n-1} a_k a_{n-k} + 2(5n-4)(5n-6) a_{n-1}.
\ee
The sequence $a_n$ is the sequence A094199 on Sloane's The On-Line Encyclopedia of Integer Sequences.
This sequence appeared in \cite{Janson, Janson-Chassaing},
see also \cite{Finch}.
By \cite[Theorem 4.2]{Janson-Chassaing},
as $n \to \infty$,
\be
a_n \sim \beta \cdot 50^{n-1} (n-1)!
\ee
for some constant $\beta$.
The constant $\beta$ has been determined by Kotesovec to be
\be
\beta = \frac{5\sqrt{15}}{2\pi^2}.
\ee
By Stirling's formula,
\be
a_n \sim \sqrt{3}  2^{n-1}  5^{2n-1/2} n^{2n-1} / (\pi \exp(2n)).
\ee
It follows that
\be
c_n \sim \sqrt{3}    5^{2n-1/2} n^{2n-1} / (\pi 12^{n}\exp(2n)).
\ee
Therefore, as $n \to \infty$,
\be
c_n^{1/n} \sim \frac{25}{12 e^2} n^2.
\ee
Therefore,
when $t_2 \neq 0$,
\be
u(t_0, t_2) = \sum_{g\geq 0} \lambda^{2g} 
c_g t_2^{3g-1} (1-2t_0t_2)^{-(5g-1)/2} + t_2^{-1}
\ee
has radius of convergence equal to zero.

\section{Discussions}

We expect to extend our construction of the mean field theory to other topological field theories. 
The following discussions are based on speculative assumptions.
We will check these assumptions in subsequent work.
The free energy of such a theory in two dimensions depends on infinitely many parameters 
$\{t^{a,n}\}_{0 \leq a \leq m, n \geq 0}$ for some fixed $m$,
where $m+1$ is the number of primary operators.
Suppose that
\be
\frac{\pd^3F_0(\bt)}{\pd t^{0,0} \pd t^{a, 0} \pd t^{b,0}} = \eta_{ab},
\ee 
where $(\eta_{ab})$ is a nondegenerate symmetric matrix. 
Denote its inverse matrix by $(\eta^{ab})$. 
For the order parameters,
as in \cite{Dijkgraaf-Witten} we take 
\be
u_a : = \lambda^2 \frac{\pd^2F}{\pd t^{0,0} \pd t^{a,0}}.
\ee
Suppose that they satisfy the integrable hierarchy 
\be
\frac{\pd u_a}{\pd t^{b, n}} = \pd_x R_{a,b; n}[u],
\ee
where $R_{a,b; n}[u]$ are some differential polynomials.
Furthermore, assume that the free energy $F$ satisfies the puncture equation of the form:
\be
\frac{\pd F}{\pd t^{0,0}}
= \sum_{a=0}^m \sum_{n=1}^\infty t^{a, n} \frac{\pd F}{\pd t^{a, n-1}} 
+ \frac{1}{2\lambda^2} \eta_{ab} t^{a,0}t^{b,0},
\ee
such that one can derive from it the string equations:
\be \label{eqn:String-General}
u_b = \sum_{n \geq 1} \sum_{c=0}^m t^{c,n} R_{b,c;n}[u] + \eta_{b0} x.
\ee
Multiply both sides by $\eta^{ab} \pd_x u_a$ and sum over repeated indices:
\be
\eta^{ab} \pd_x u_a \cdot u_b = \sum_{n \geq 1} \sum_{c=0}^m \eta^{ab} t^{c,n} \pd_x u_a \cdot R_{b,c; n}[u] 
+ x \pd_x u_0.
\ee
Suppose that after integration one has
\be
\half \eta^{ab} u_a u_b = \sum_{n \geq 1} \sum_{c=0}^m \eta^{ab} t^{c,n} T_{a, b,c; n}[u]
+ x u_0 - \lambda^2 \frac{\pd F}{\pd x}.
\ee
Then one gets:
\ben
F = \lambda^{-2} \int \biggl(\sum_{n \geq 1} \sum_{c=0}^m \eta^{ab} t^{c,n} T_{a, b, c; n}[u]
+ x u_0 - \half \eta^{ab} u_a u_b \biggr) dx + F|_{t^{0,0} = 0},
\een
where $u$ satisfies the string equations \eqref{eqn:String-General}.
Next we define the action functional to be
\ben
S = \lambda^{-2} \int \biggl(\sum_{n \geq 1} \sum_{c=0}^m \eta^{ab} t^{c,n} T_{a, b, c; n}[u]
+ x u_0 - \half \eta^{ab} u_a u_b \biggr) dx + F|_{t^{0,0} = 0},
\een
and our final assumption is that the system of Euler-Lagrange equations
\be
\frac{\delta S}{\delta u_a} = 0, \;\;\; a =0, \dots, m, 
\ee
is equivalent to the system of string equations \eqref{eqn:String-General}.
When all our assumptions are met,
we then arrive at a one-dimensional mean field theory for the original theory.
We conjecture this is the case for the theories arising from semisimple Frobenius manifolds
\cite{Dubrovin-Zhang}.

Another direction for possible generalizations is to find $(n-1)$-dimensional mean field theory for
topological $n$-dimensional gravity.
In \cite{Zhou2} we have studied topological 1D gravity by a 0-dimensional mean field theory.
We have seen in Section 3 that the genus zero part of 
the Lagrangian density for the topological 2D gravity 
is the  Lagrangian density for the topological 1D gravity.
Furthermore, 
in both cases the system of equations of motion has only one formal solution. 
It will be very interesting to generalize these to higher dimensions.

\vspace{.2in}
{\em Acknowledgements}.
This research is partially supported by NSFC grant 11171174.
The author thanks Siqi Liu for sharing with him his notes on DVV relations.
He also thanks Si Li for discussions on relations between topological 1D and 2D gravity
for a related problem.

\end{document}